\documentclass[12pt, reqno]{amsart}
\usepackage{amssymb,amsthm,amsmath}
\usepackage[numbers,sort&compress]{natbib}
\usepackage{amssymb,amsmath}
\usepackage{amsfonts,dsfont}
\usepackage{mathrsfs}
\usepackage{latexsym}
\usepackage{amssymb}
\usepackage{amsthm}
\usepackage{indentfirst}
\usepackage{hyperref}
\usepackage{color}
\date{\today
}

\let\oldsection\section
\renewcommand\section{\setcounter{equation}{0}\oldsection}

\newtheorem{theorem}{Theorem}[section]
\newtheorem{lemma}{Lemma}[section]
\newtheorem{proposition}{Proposition}[section]
\newtheorem{definition}{Definition}[section]

\newtheorem{remark}{Remark}[section]

\allowdisplaybreaks
\begin{document}

\title[ nonlinear wave propagation in the troposphere]{Global well-posedness of strong solutions to a model for the morning glory cloud}

\author{Jinkai~Li}
\address[Jinkai~Li]{South China Research Center for Applied Mathematics and Interdisciplinary Studies, School of Mathematical Sciences,
South China Normal University, Guangzhou 510631, China and Guangdong Key Laboratory of Applied Mathematics and Artificial Intelligence}
\email{jklimath@gmail.com; jklimath@m.scnu.edu.cn}

\author{Yuan~Ma}
\address[Yuan~Ma]{School of Mathematical Sciences, South China Normal University,
Guangzhou, 510631, China}
\email{657861189@qq.com}

\author{Dong~Wang}
\address[Dong~Wang]{South China Research Center for Applied Mathematics and Interdisciplinary Studies, School of Mathematical Sciences,
South China Normal University, Guangzhou 510631, China}
\email{WDongter@m.scnu.edu.cn}


%%% ----------------------------------------------------------------------

\begin{abstract}
In this paper, we investigate the global well-posedness of strong solution to a model recently derived by Constantin-Johnson \cite{CaJo} which describes the nonlinear wave propagation in the troposphere, especially for the morning glory cloud. Assuming that the initial velocity $v_0\in H^1$ and the thermodynamic forcing term $K\in L^2(0,T;L^2)$, we show that there exists a unique global strong solution to the initial boundary value problem of this system. Similar result was only known before for sufficiently small initial data in the existing literature.
\end{abstract}

%%% ----------------------------------------------------------------------
\maketitle

%\tableofcontents

\section{Introduction}
\label{sec1}
\allowdisplaybreaks
The morning glory is a spectacular cloud pattern
composed of a tubular cloud, or succession of such roll
clouds, especially in Australia. A mathematical model for the morning glory cloud pattern was derived by Constantin-Johnson \cite{CaJo}, which describes the nonlinear wave propagation in the troposphere. Due to typically stretching from horizon to horizon, the motion behaves essentially in a two dimensional way. More details about this significant nature phenomenon can be found in \cite{Bir1,Chtr1,Cred1}.
Following \cite{MatBog} and \cite{CaJo}, we consider the following system:
\begin{align}
\label{eqpeini}
\begin{cases}
&\partial_tv+v\partial_xv+w\partial_zv-\mu\Delta v+\alpha v+\beta w+K=0,  \\
&\partial_x v+\partial_zw=0,
\end{cases}
\end{align}
where the horizontal velocity $v$ and the vertical velocity $w$ are unknowns. Here $\alpha$ and $\beta $ are constants and $K$ is a given forcing term.

Several works have been done on the morning glory system \eqref{eqpeini}. Concerning this system, Constantin-Johnson \cite{CaJo} discovered a special oscillatory solution, and studied the traveling wave solutions in \cite{CaJoa1}. Matioc-Roberti \cite{MatBog} established the global existence of weak solutions and local well-posedness of strong solutions. Precisely, they considered \eqref{eqpeini} in the infinite channel $\mathbb{R}\times I$, subject to the boundary conditions $v|_{z=0}=w|_{z=0}=v|_{z=1}=0$. They established the global existence of weak solutions with $v_0\in L^2$ and $K\in L^2(0,T;L^2)$ by employing a two-step approximation approach based on the Galerkin scheme. Besides, by formulating the model as a quasilinear parabolic evolution system in an appropriate functional analytic framework, then proved that \eqref{eqpeini} has a unique local-in-time strong solution. Global existence of strong solutions was obtained by Alonso-Or$\acute{a}$n and Granero-Belinch$\acute{o}$n \cite{Alonra}; however, they required that $\|v_0\|_{L^\infty}$ is sufficiently small.

To our best knowledge, global existence of strong solutions to system \eqref{eqpeini} with large initial data is unavailable in the existing literature. The aim of this paper is to address this problem, that is to establish the global existence of strong solutions to system \eqref{eqpeini} with arbitrary large initial data.

In this paper, we consider \eqref{eqpeini} in the domains $\Omega=\mathbb{T}\times(0,1)$ subject to the following boundary and initial conditions:
\begin{align}
  &v|_{z=1}=\partial_z v|_{z=0}=w|_{z=0}=0,\label{eq5}\\
  &v \,\text{is periodic in} \,x,\label{eq6}\\
  &v|_{t=0}=v_0.\label{eq7}
\end{align}

It follows from $\eqref{eqpeini}_2$ and \eqref{eq5} that
\begin{align}
\label{omegaas}
&\displaystyle w=-\int_0^z\partial_xv(x,\xi,t)d\xi.
\end{align}

One may observe that system \eqref{eqpeini} is very similar to the primitive equations (PEs) which in the two dimensional case read as
\begin{align*}
\begin{cases}
&\partial_tv+v\partial_xv+w\partial_zv-\Delta v+\partial_xp=0,  \\
& \partial_zp=0,\\
&\partial_x v+\partial_zw=0,
\end{cases}
\end{align*}
where $(v,w)$ is the velocity and $p$ is the pressure. An essential difference between \eqref{eqpeini} and the above PEs is the presence of the pressure term. Due to the presence of the pressure term, one can complement the boundary conditions $w|_{z=0}=w|_{z=1}=0$ to the PEs. However, for system \eqref{eqpeini}, besides the boundary conditions on $v$, one can only complement it with one side boundary condition for $w$, that is with either $w|_{z=0}=0$ or $w|_{z=1}=0$ but not with both; otherwise, it may be overdetermined. One may refer to \cite{Lions1,Lions2,Lions3,J.Li4,DFA,IYWM,TTM,LiYuan,Cao1,GMK,IKMZ,Boling,MHTKa,MHATK,Cao2,J.Li5,J.Li6,J.Li1,J.Li7, J.Li8,J.LiW} for the mathematical results on the PEs.

Throughout this paper, $L^q(\Omega)$ and $W^{m,q}(\Omega)$ denote the standard Lebesgue and Sobolev spaces, repectively. For $q=2$, we use $H^m$ to replace $W^{m,2}$. We always use $\|f\|_p$ to denote the $L^p(\Omega)$ norm of $f$.

\begin{definition}\label{def1}
Given a positive time $T$ and a function $ v_0 \in H^1(\Omega) $ which is periodic in $x$ and satisfies $v_0|_{z=1}=0$.
 $v$ is called a strong solution to system
\eqref{eqpeini}, subject to \eqref{eq5}--\eqref{eq7}, on
$\Omega\times(0, T)$, if

(i) $v$ has the regularities
\begin{eqnarray*}
  &&v \in C([0,T];H^1(\Omega))\cap L^2(0,T; H^2(\Omega)),\quad
\partial_t v\in L^2(0,T;L^2(\Omega));
\end{eqnarray*}

(ii) $v$ satisfies \eqref{eqpeini} a.e.\,in $\Omega \times (0,T)$, with $w$ given by \eqref{omegaas}, and satisfies the initial and boundary conditions (\ref{eq5})--(\ref{eq7}).
\end{definition}

\begin{definition}\label{def2}
 $v$ is called a global strong solution to system \eqref{eqpeini}, subject to (\ref{eq5})--(\ref{eq7}), if it is a strong solution
on $\Omega\times(0,T)$, for any $T \in (0,\infty)$.
\end{definition}

Our main result is stated as follows
\begin{theorem}\label{the1}
Given $K\in L^2_{loc}([0,\infty);L^2(\Omega))$. Assume that $v_0\in H^1(\Omega)$ is periodic in $x$ and satisfies $v_0|_{z=1}=0$. Then, there exists a unique global strong solution $v$ to system \eqref{eqpeini}, subject to \eqref{eq5}--\eqref{eq7}.
\end{theorem}
\begin{remark}
(i)Theorem \ref{the1} also holds when the boundary condition $\partial_zv|_{z=0}=0$ in \eqref{eq5} is replaced by $v|_{z=0}$.\\
(ii)By adopting the arguments in this paper, one can also show that  Theorem \ref{the1} still holds for the case $\Omega:=\mathbb{R}\times (0,1)$.
\end{remark}
\begin{remark}
Inspired by the theory of the 2D primitive equations,
Holst-Rademacher \cite{holrad} considered the following system
\begin{align*}
\begin{cases}
&\partial_tv+v\partial_xv+w\partial_zv-\mu\Delta v+\alpha v+\beta w+K+\partial_xp=0,  \\
&\partial_x v+\partial_zw=0,
\end{cases}
\end{align*}
subject to
\begin{align*}
\begin{cases}
&v, w \text{ are periodic in } x,\\
&v|_{z=0}=w|_{z=0}=v|_{z=1}=w|_{z=1}=0,\\
&v|_{t=0}=v_0,
\end{cases}
\end{align*}
and obtained the global existence of strong solutions.
Note that this modified system has a structural relation to the 2D
primitive equations, for which global well-posedness of strong soltions has been known for long time.
\end{remark}

The main difficulty to prove Theorem \ref{the1} is to get the a priori $L^\infty(0,T;H^1)$ estimate on $v$. Since system \eqref{eqpeini} seems like the two dimensional PEs, one may try to use the same idea as for the PEs to deal with \eqref{eqpeini}, subject to \eqref{eq5}--\eqref{eq7}. In particular, one may try to carry out either the a priori $\|v\|_{L^\infty(0,T;H^1)}$ or $\|\partial_z v\|_{L^\infty(0,T;L^2)}$ after getting the $L^\infty(0,T;L^q)$ type estimate on $v$, which is now one of the standard ideas of getting the a priori estimates for the PEs. Unfortunately, it is not the case for our problem. In fact, due to the absence of boundary condition for $w$ on $\{z=1\}$, if following the arguments as for the PEs to get the a priori $\|v\|_{L^\infty(0,T;H^1)}$ or $\|\partial_zv\|_{L^\infty(0,T;L^2)}$, one will encounter a boundary term $\int_0^1w|\partial_zv|^2dx|_{z=0}^{z=1}$ which comes from integrating by parts to the term $\int_\Omega w\partial_zv\partial_z^2vdxdz$. Note that this boundary term vanishies in the case of the PEs, as the boundary conditions $w|_{z=0}=w|_{z=1}=0$ are complemented to the PEs. To overcome this difficult, following the idea in Li-Wang \cite{LiWangz}, we first carry out the a priori $L^\infty(0,T;L^2)\cap L^2(0,T;H^1)$ estimates on $v^2$. Based on this, by deriving the dynamical equation for $w$, see \eqref{eqomega1} in the below, and performing the energy estimate to \eqref{eqomega1}, we get the $L^\infty(0,T;L^2)\cap L^2(0,T;H^1)$ estimate on $w$. Then, one can successfully deal with the term $\int_\Omega w\partial_zv\partial_z^2vdxdz$ as
\begin{align*}
\left|\int_\Omega w\partial_zv\partial_z^2vdxdz\right|\leq C\|w\|_2^2\|\nabla w\|_2^2\|\partial_zv\|_2^2+\text{ other term},
\end{align*}
and further get the desired a priori $L^\infty(0,T;H^1)\cap L^2(0,T;H^2)$ estimate on $v$.

The rest of this paper is arranged as follows: in the next section, Section 2, we
collect some preliminary results which will be used in the subsequent sections; in Section
3, we establish the local well-posedness to system \eqref{eqpeini}, subject to \eqref{eq5}--\eqref{eq7}. In section 4, we establish the a priori estimates to strong
solutions, especially the $L^\infty(0,T; H^1)$-estimate on $v$; in Section 5, we give the prove of Theorem 1.1.
Throughout this paper, the letter C denotes a general positive constant which may
vary from line to line.
\section{Preliminaries}\label{sec2}

In this section, we state several preliminary lemmas which will be used
in the rest of this paper.
\begin{lemma}\label{lem1a}
It holds that
\begin{align*}
&\left\|v^3\right\|_{L^2_xL^1_z}\leq C\left(\|v\partial_xv\|_2^\frac{1}{4}+\left\|v\right\|_4^\frac{1}{2}\right)\|v\|_4^\frac{3}{2}\left(\|\partial_xv\|_2^\frac{1}{4}+\|v\|_2^\frac{1}{4}\right)\|v\|_2^\frac{3}{4},\\
&\left\|v^3\right\|_{2}\leq C\left(\|v\nabla v\|_2^\frac{1}{2}+\|v\|_4\right)\|v\|_4^2,
\end{align*}
where $C$ is an absolute positive constant.
\end{lemma}
\begin{proof}
It follows from the Minkowski, H\"{o}lder, and Gagliardo-Nirenberg inequalities that
\begin{align*}
\left\|v^3\right\|_{L^2_xL^1_z}=&\left\|\left\|v^3\right\|_{L_z^1}\right\|_{L^2_x}
\leq\left\|\left\|v^2\right\|_{L^2_z}\left\|v\right\|_{L^2_z} \right\|_{L^2_x} \leq \left\|\left\|v^2\right\|_{L^2_z}\right\|_{L^4_x}\left\|\left\|v\right\|_{L^2_z} \right\|_{L^4_x}\\
=&\left\|v^2\right\|_{L^4_xL^2_z}\|v\|_{L^4_xL^2_z}\leq \left\|v^2\right\|_{L^2_zL^4_x}\|v\|_{L^2_zL^4_x}
=\left\|\left\|v^2\right\|_{L^4_x}\right\|_{L^2_z}\left\|\|v\|_{L^4_x}\right\|_{L^2_z}\\
\leq& C\left\|\left\|v^2\right\|_{L^2_x}^\frac{3}{4}\left\|v^2\right\|_{H^1_x}^\frac{1}{4}\right\|_{L^2_z}
\left\|\|v\|_{L^2_x}^\frac{3}{4}\|v\|_{H^1_x}^\frac{1}{4}\right\|_{L^2_z}\\
\leq&C\|v^2\|_2^\frac{3}{4}\left(\|v\partial_xv\|_2^\frac{1}{4}+\left\|v^2\right\|_2^\frac{1}{4}\right)\|v\|_2^\frac{3}{4}\left(\|\partial_xv\|_2^\frac{1}{4}+\|v\|_2^\frac{1}{4}\right)\\
=& C\|v\|_4^\frac{3}{2}\left(\|v\partial_xv\|_2^\frac{1}{4}+\left\|v\right\|_4^\frac{1}{2}\right)\|v\|_2^\frac{3}{4}\left(\|\partial_xv\|_2^\frac{1}{4}+\|v\|_2^\frac{1}{4}\right)
\end{align*}
and
\begin{align*}
\left\|v^3\right\|_2
\leq\left\|v^2\right\|_4\|v\|_4\leq C\left\|v^2\right\|_{H^1}^\frac{1}{2}\left\|v^2\right\|_2^\frac{1}{2}\|v\|_4\leq C\left(\|v\nabla v\|_2^\frac{1}{2}+\|v\|_4\right)\|v\|_4^2,
\end{align*}
proving the conclusion.
\end{proof}

 We also need the following Aubin-Lions Lemma.
\begin{lemma}[Aubin-Lions Lemma, See Corollary 4 in \cite{JSimon}]\label{lemma7}
Assume that $X$, $B$ and $Y$ are three Banach spaces, with $X\hookrightarrow\hookrightarrow B\hookrightarrow Y$. Then the following two items hold:

(i) If $\mathcal{F}$ is a bounded subset of $L^p(0,\mathcal{T};X)$, where $1\leq p< \infty$,
and $\frac{\partial \mathcal{F}}{\partial t} :=\left\{\frac{\partial f}{\partial t} \big| f \in \mathcal{F}\right\}$ is bounded in
$L^1(0,\mathcal{T};Y)$, then $\mathcal{F}$ is relatively compact in $L^p(0,\mathcal{T};B)$;

(ii) if $\mathcal{F}$ is a bounded subset of $L^{\infty}(0,\mathcal{T};X)$, and $\frac{\partial \mathcal{F}}{\partial t}$
is bounded in $L^r(0,\mathcal{T};Y)$, where $r>1$, then $\mathcal{F}$ is relatively compact in $C([0,\mathcal{T}];B)$.
\end{lemma}

\section{Local well-posedness}
This section is devoted to proving the local well-posedness for system \eqref{eqpeini}, subject to \eqref{eq5}--\eqref{eq7}.
\subsection{Local existence}
\begin{proposition}
\label{prolocal}
Given $K\in L^2_{loc}([0,\infty);L^2(\Omega))$. Assume that $v_0\in H^1(\Omega)$ such that $v_0$ is periodic in $x$ and $v_0|_{z=1}=0$. Then, there is a positive time $T$ depending only on $\|v_0\|_{H^1(\Omega)}$ and $K$, such that system \eqref{eqpeini}, subject to \eqref{eq5}--\eqref{eq7} has a unique solution $v$ on $\Omega\times (0,T)$, satisfying
\begin{align*}
&v\in C([0,T]; H^1(\Omega))\cap L^2(0,T; H^2(\Omega)), \quad \partial_tv\in L^2\left(\Omega\times(0,T)\right).
\end{align*}

\end{proposition}
We will apply the contractive mapping principle to prove this proposition.

Define the space $X_{T}:=\{\phi \in L^\infty(0,T;H^1)\cap L^2(0,T;H^2), \phi|_{z=1}=\partial_z\phi|_{z=0}=0\}$, and $\|f\|_{X_{T}}:=\|f\|_{L^\infty(0,T;H^1)}+\|f\|_{L^2(0,T;H^2)}$. Then, $X_{T}$ is a Banach space.

Define a map $\mathcal{F}$: $X_{T}\rightarrow X_{T}$ as follows. For any $v\in X_{T}$,
\begin{align*}
U=\mathcal{F}(v)
\end{align*}
is the unique solution to
\begin{align}\label{eqUloc}
&\partial_t U-\mu\triangle U=-v\partial_xv-w\partial_zv-\alpha v-\beta w-K=:G(v),
\end{align}
with $w$ given by \eqref{omegaas},
subject to
\begin{align}
\label{eqUbdd}
\begin{cases}
&U|_{z=1}=\partial_zU|_{z=0}=0,\\
&U\, \text{is periodic in}\, x,\\
&U|_{t=0}=v_0.
\end{cases}
\end{align}
Since $v\in X_{T}$ and $K\in L^2(0,T;L^2(\Omega))$, one can check easily that $G(v)\in L^2(0,T; L^2(\Omega))$. Then, the standard parabolic theory implies that
\begin{align}\label{estilocU}
&U\in C([0,T]; H^1(\Omega))\cap L^2(0,T;H^2(\Omega)),\quad\partial_tU\in L^2(0,T;L^2(\Omega)).
\end{align}
Therefore, $\mathcal{F}$ is well-defined on $X_T$. In order to show Proposition \ref{prolocal}, it suffices to find a fixed point of $\mathcal{F}$ in $X_{T}$.
\begin{proposition}
\label{minusloc}
Given arbitrary positive number $M\geq 1$ and time $T\leq 1$ and set $\mathbb{B}_{M,T}:=\left\{v\in X_{T}\Big| \|v\|_{X_{T}}\leq M\right\}$. Then, for any $v_1,v_2\in \mathbb{B}_{M,T}$, it holds that
\begin{align*}
\|\mathcal{F}(v_1)-\mathcal{F}(v_2)\|_{X_{T}}\leq
C_0\left(MT^\frac{1}{4}+T^\frac{1}{2}\right)\|v_1-v_2\|_{X_{T}},
\end{align*}
where $C_0$ is a positive constant depending only on $\mu$, $\alpha$, and $\beta$.
\end{proposition}
\begin{proof}
Let $w_i$ be given by \eqref{omegaas} related to $v_i$, $i=1,2$.
Set $U=U_1-U_2$, $v=v_1-v_2$, and $w=w_1-w_2$. Then, it is clear that
$w$ satisfies \eqref{omegaas} and
\begin{align}
\label{eqlocminus1}
\begin{cases}
&\partial_tU-\mu\triangle U=G(v_1)-G(v_2),\\
&U|_{z=1}=\partial_zU|_{z=0}=0, \quad U \text{ is periodic in } x,\\
&U|_{t=0}=0.
\end{cases}
\end{align}
Note that
\begin{align}
\|f\|_{L_z^2L_x^4}=&\left\|\|f\|_{L^4_x}\right\|_{L^2_z}\leq C\left\|\|f\|_{L^2_x}^\frac{3}{4}\|f\|_{H^1_x}^\frac{1}{4}\right\|_{L^2_z}
\nonumber\\
\leq& C\left\|\|f\|_{L^2_x}\right\|_{L^2_z}^\frac{3}{4}
\left\|\|f\|_{H^1_x}\right\|_{L^2_z}^\frac{1}{4}\leq C\|f\|_2^\frac{3}{4}\|f\|_{H^1}^\frac{1}{4}.\label{fz2x4*}
\end{align}
Thanks to this and using \eqref{omegaas},
it follows from the H\"{o}lder, Gagliardo-Nirenberg and Minkowski inequalities that
\begin{align*}
&\|G(v_1)-G(v_2)\|_2=\left\|v_1\partial_xv+v\partial_xv_2+w_1\partial_zv+w\partial_zv_2+\alpha v+\beta w\right\|_2\\
\leq& \|v_1\|_4\|\partial_xv\|_4+\|v\|_4\|\partial_xv_2\|_4+\|w_1\|_{L_x^4L^\infty_z}\|\partial_z v\|_{L^4_xL_z^2}\\
&+\|w\|_{L_x^4L^\infty_z}\|\partial_z v_2\|_{L^4_xL_z^2}+C\|v\|_2+C\|w\|_2\\
\leq& \|v_1\|_4\|\partial_xv\|_4+\|v\|_4\|\partial_xv_2\|_4+\|\partial_xv_1\|_{L^4_xL^1_z}\|\partial_zv\|_{L^4_xL^2_z}\\
&+\|\partial_xv\|_{L^4_xL^1_z}\|\partial_zv_2\|_{L^4_xL^2_z}+C\|v\|_2+C\|\partial_xv\|_2\\
\leq& \|v_1\|_4\|\partial_xv\|_4+\|v\|_4\|\partial_xv_2\|_4+C\|\partial_xv_1\|_{L_z^2L^4_x}\|\partial_zv\|_{L^2_zL^4_x}\\
&+C\|\partial_xv\|_{L^2_zL^4_x}\|\partial_zv_2\|_{L^2_zL^4_x}+C\|v\|_2+C\|\partial_xv\|_2\\
\leq& C\|v_1\|_{H^1}\|v\|_{H^1}^\frac{1}{2}\| v\|_{H^2}^\frac{1}{2}+C\|v\|_{H^1}\|v_2\|_{H^1}^\frac{1}{2}\| v_2\|_{H^2}^\frac{1}{2}+C\|v_1\|_{H^2}^\frac{1}{4}\|v_1\|_{H^1}^\frac{3}{4}\|v\|_{H^2}^\frac{1}{4}\|v\|_{H^1}^\frac{3}{4}\\
&+C\|v\|_{H^2}^\frac{1}{4}\|v\|_{H^1}^\frac{3}{4}\|v_2\|_{H^2}^\frac{1}{4}\|v_2\|_{H^1}^\frac{3}{4}
+C\|v\|_{H^1},
\end{align*}
from which, recalling that $\|v_i\|_{X_T}\leq M$, one gets
\begin{align}
\int_0^{T}\|G(v_1)-G(v_2)\|_2^2dt\leq C\left(M^2T^\frac{1}{2}\|v\|_{X_{T}}^2+T\|v\|_{X_{T}}^2\right).\label{eqG1-2}
\end{align}
Therefore, multiplying $\eqref{eqlocminus1}_1$ by $U-\triangle U$ and integration by part, it follows from the H\"{o}lder inequality that
\begin{align*}
&\frac{1}{2}\frac{d}{dt}\left(\|U\|_2^2+\|\nabla U\|_2^2\right)+\mu\left(\|\nabla U\|_2^2+\|\triangle U\|_2^2\right)\\
\leq&\|G(v_1)-G(v_2)\|_2\|U\|_2+\|G(v_1)-G(v_2)\|_2\|\triangle U\|_2\\
\leq& \frac{\mu}{2}\|\triangle U\|_2^2+C\left(\|G(v_1)-G(v_2)\|_2^2+\|U\|_2^2\right),
\end{align*}
from which, by the Gr\"{o}nwall inequality and \eqref{eqG1-2}, one obtains
\begin{align*}
\|U\|_{X_{T}}^2\leq Ce^{CT}\int_0^{T}\|G_1(v_1)-G_2(v_2)\|_2^2dt\leq C \left(M^2T^\frac{1}{2}+T\right)\|v\|_{X_{T}}^2,
\end{align*}
proving the conclusion.
\end{proof}
\begin{proposition}\label{loc2M}
Let $C_0$ be the constant as in Proposition \ref{minusloc}. Then, there exists a positive constant $M_0$ depending only $v_0$ and $K$ such that $\|\mathcal{F}(v)\|_{X_{T_0}}\leq 2M_0$ for any $\|v\|_{X_{T_0}}\leq 2M_0$, where $T_0=\min\left\{1, \frac{1}{(4C_0M_0)^4}\right\}$.
\end{proposition}
\begin{proof}
The standard theory of parabolic equations provides that there exists a constant $M_0\geq 1$, depending only on $v_0$ and $K$, such that
\begin{align*}
\|\mathcal{F}(0)\|_{X_1}\leq M_0.
\end{align*}
Let $C_0$ be the constant stated in Proposition \ref{minusloc} and denote $T_0=\min\left\{1,\frac{1}{(4C_0M_0)^4}\right\}$. Take arbitrary $v\in X_{T_0}$ satisfying $\|v\|_{X_{T_0}}\leq 2M_0$.
Note that $\|\mathcal{F}(0)\|_{X_{T_0}}\leq \|\mathcal{F}(0)\|_{X_1}\leq M_0$. It follows from Proposition \ref{minusloc} that
\begin{align*}
\|\mathcal{F}(v)\|_{X_{T_0}}\leq & \|\mathcal{F}(0)\|_{X_{T_0}}+\|\mathcal{F}(v)-\mathcal{F}(0)\|_{X_{T_0}}
\leq M_0+C_0\left(M_0T_0^\frac{1}{4}+T_0^\frac{1}{2}\right)\|v\|_{X_{T_0}}\\
\leq& M_0+2C_0M_0^2T_0^\frac{1}{4}+2C_0M_0T_0^\frac{1}{2}\leq M_0+4C_0M_0^2T_0^\frac{1}{4}\leq 2M_0,
\end{align*}
proving the conclusion.
\end{proof}
\begin{proof}
\bf The proof of Proposition \ref{prolocal}. \rm
{\bf Existence.}
Let $M_0$ and $T_0$ be as in Proposition \ref{loc2M} and denote $\mathbb{B}_{2M_0, T_0}=\left\{v\in X_{T_0}\left|\|v\|_{X_{T_0}}\leq 2M_0\right.\right\}$. Then it follows from Proposition \ref{minusloc} and Proposition \ref{loc2M} that
\begin{align*}
&\mathcal{F}:\mathbb{B}_{2M_0, T_0}\rightarrow \mathbb{B}_{2M_0, T_0},\\
&\|\mathcal{F}(v_1)-\mathcal{F}(v_2)\|_{X_{T_0}}\leq \frac{1}{2}\|v_1-v_2\|_{X_{T_0}},
\end{align*}
for any $v_1, v_2\in \mathbb{B}_{2M_0, T_0}$. Therefore, $\mathcal{F}$ is a contractive mapping on $\mathbb{B}_{2M_0, T_0}$. By the contractive mapping principle, there is a unique fixed point $v$ of $\mathcal{F}$ on $\mathbb{B}_{2M_0, T_0}$, which by definition of $\mathcal{F}$ is a solution to system \eqref{eqpeini}, subject to \eqref{eq5}--\eqref{eq7}. The desired regularities of $v$ follow from
\eqref{estilocU}.
%\end{proof}
%\subsection{continuous with initial data}\label{conlo}
%\begin{proof}

{\bf Uniqueness.}
Let $v_1$ and $v_2$ be two strong solutions to system \eqref{eqpeini}, subject to \eqref{eq5}--\eqref{eq7}, with the same initial data. Let $w_i$ be given by \eqref{omegaas} related to $v_i$, $i=1,2$. Denote $v=v_1-v_2$ and $w=w_1-w_2$. Then, it is clear that
\begin{align}
\label{locmi1}
\begin{cases}
&\partial_tv-\mu\triangle v+\alpha v+\beta w+v_1\partial_xv+v\partial_xv_2+w_1\partial_zv+w\partial_zv_2=0,\\
&v|_{z=1}=\partial_zv|_{z=0}=0,w|_{z=0}=0,\\
&v,w \,\text{ are periodic in }\, x,\\
&v|_{t=0}=0.
\end{cases}
\end{align}
Multiplying $\eqref{locmi1}_1$ with $v$, due to the boundary conditions $\eqref{locmi1}_2$, and using the Minkowski, H\"{o}lder, Gagliardo-Nirenberg, and Young inequalities, it follows from integrating by part and \eqref{fz2x4*} that
\begin{align*}
&\frac{1}{2}\frac{d}{dt}\|v\|_2^2+\mu\|\nabla v\|_2^2
=  -\int_\Omega \left(\alpha v+\beta w+v\partial_xv_2+w\partial_zv_2\right)vdxdz\\
\leq &
C\left(\|v\|_2^2+\|w\|_2\|v\|_2+\|v\|_4^2\|\partial_xv_2\|_2+\|w\|_{L^\infty_zL^2_x}\|\partial_zv_2\|_{L^2_zL^4_x}\|v\|_{L^2_zL^4_x}\right)\\
\leq&
C\left(\|v\|_2^2+\|\partial_xv\|_2\|v\|_2+\|v\|_4^2\|\partial_xv_2\|_2+\|\partial_xv\|_{2}\|\partial_zv_2\|_{L^2_zL^4_x}\|v\|_{L^2_zL^4_x}\right)\\
\leq & C\left(\|v\|_2^2+\|v\|_2\|\partial_xv\|_2+\|v\|_2\| v\|_{H^1}\|\partial_xv_2\|_2+\|\partial_xv\|_2\|v\|_{H^1}^\frac{1}{4}\|v\|_2^\frac{3}{4}\|\partial_zv_2\|_2^\frac{3}{4}\|\partial_zv_2\|_{H^1}^\frac{1}{4}\right)\\
\leq& \frac{\mu}{2}\|\nabla v\|_2^2+C\left(1+\|\partial_xv_2\|_2^2+\|\partial_zv_2\|_2^2\|\partial_zv_2\|_{H^1}^\frac{2}{3}\right)\|v\|_2^2,
\end{align*}
from which, the conclusion follows by applying the Gr\"{o}nwall inequality.
\end{proof}
\section{A priori estimates}\label{localex3}
Throughout this section, we assume that $v$ is a strong solution to system \eqref{eqpeini}, subject to \eqref{eq5}--\eqref{eq7}, on the time interval $(0, T)$, for a finite positive time $T$.
\begin{proposition}
\label{provL2}
It holds that
\begin{align*}
\sup_{0\leq t\leq T}\|v\|_2^2+\int_0^T\|\nabla v\|_2^2dt\leq C,
\end{align*}
where $C$ is a positive constant depending only $\|v_0\|_2$, $\alpha$, $\beta$, $\mu$, $\|K\|_{L^2(0, T;L^2)}$, and $T$.
\end{proposition}
\begin{proof}
Multiplying $\eqref{eqpeini}_1$ by $v$, and using the boundary conditions \eqref{eq5}--\eqref{eq6}, it follows from integrating by parts, the H\"{o}lder and Young inequalities that
\begin{align*}
\frac{1}{2}\frac{d}{dt}\|v\|_2^2+\mu\|\nabla v\|_2^2\leq & |\alpha|\|v\|_2^2+|\beta|\|\partial_x v\|_2\|v\|_2+\|K\|_2\|v\|_2\\
\leq& \frac{\mu}{2}\|\nabla v\|_2^2+C\left(\|v\|_2^2+\|K\|_2^2\right),
\end{align*}
from which, the conclusion follows by applying the Gr\"{o}nwall inequality.
\end{proof}
Following the idea in Li-Wang \cite{LiWangz},
before establishing the $L^\infty_tH^1$ estimate for $v$, we carry out the $L^\infty_tL^4$ and $L^\infty_tL^2$ estimates for $v$ and $\omega$, respectively.
\begin{proposition}
\label{provL4}
It holds that
\begin{align*}
\sup_{0\leq t\leq T}\|v\|_4^4+\int_0^T\left\||v|\nabla v\right\|_2^2dt\leq C,
\end{align*}
where $C$ is a positive constant depending only $\|v_0\|_2$, $\|v_0\|_{4}$, $\alpha$, $\beta$, $\mu$, $\|K\|_{L^2(0,T; L^2)}$, and $T$.
\end{proposition}
\begin{proof}
Multiplying $\eqref{eqpeini}_1$ by $|v|^2v$, using the boundary conditions \eqref{eq5}--\eqref{eq6}, it follows from integrating by parts, Lemma \ref{lem1a}, and the Minkowski, H\"{o}lder, Gagliardo-Nirenberg, and Young inequalities that
\begin{align*}
&\frac{1}{4}\frac{d}{dt}\|v\|_4^4+3\mu\left\||v|\nabla v\right\|_2^2=-\int_\Omega\left(\alpha v+\beta w+K\right)|v|^2vdxdz\\
\leq & |\alpha|\|v\|_4^4+|\beta|\|w\|_{L^2_xL^\infty_z}\left\|v^3\right\|_{L^2_xL^1_z}+\|K\|_2\left\|v^3\right\|_2\\
\leq & |\alpha|\|v\|_4^4+|\beta|\|\partial_xv\|_2\left\|v^3\right\|_{L^2_xL^1_z}+\|K\|_2\left\|v^3\right\|_2\\
\leq& C\|v\|_4^4+C\|\partial_xv\|_2\left(\|v\partial_xv\|_2^\frac{1}{4}+\left\|v\right\|_4^\frac{1}{2}\right)\|v\|_4^\frac{3}{2}\left(\|\partial_xv\|_2^\frac{1}{4}+\|v\|_2^\frac{1}{4}\right)\|v\|_2^\frac{3}{4}
\\
&+C\|K\|_2\left(\|v\nabla v\|_2^\frac{1}{2}+\|v\|_4\right)\|v\|_4^2\\
\leq& C\|v\|_4^4+C\|v\partial_x v\|_2^\frac{1}{4}\|v\|_4^\frac{3}{2}\|v\|_2^\frac{3}{4}\|\partial_xv\|_2^\frac{3}{4}\|\partial_xv\|_2^\frac{1}{4}\left(\|\partial_xv\|_2+\|v\|_2\right)^\frac{1}{4}\\
&+C\|\partial_xv\|_2\|v\|_4^2\left(\|\partial_xv\|_2^\frac{1}{4}+\|v\|_2^\frac{1}{4}\right)\|v\|_2^\frac{3}{4}
+C\|K\|_2\left(\|v\nabla v\|_2^\frac{1}{2}+\|v\|_4\right)\|v\|_4^2\\
\leq& \mu\|v\nabla v\|_2^2+C\left(\|\partial_x v\|_2^2\|v\|_2^2+\|\partial_xv\|_2^2+\|K\|_2^2+1\right)\|v\|_4^4+C\left(\|v\|_{H^1}^2+\|K\|_2^2+1\right),
\end{align*}
which implies the conclusion by applying the Gr\"{o}nwall inequality and Proposition \ref{provL2}.
\end{proof}
Recalling \eqref{omegaas}, i.e.,
$
\displaystyle w=-\int_0^z\partial_xv(x,\xi,t)d\xi.
$
One can verify from \eqref{eqpeini} by direct calculations that
\begin{align}
\label{eqomega1}
\begin{cases}
&\partial_tw-\mu \triangle w+2\int_0^z\partial_x(v\partial_xv)d\xi+\partial_x(w v)+\alpha w+\beta \int_0^z\partial_xw d\xi+\int_0^zK_xd\xi=0,\\
&w|_{z=0}=0, \,\partial_zw|_{z=1}=0,\, w \text{ is periodic in }x, \, w|_{t=0}=w_0.
\end{cases}
\end{align}
\begin{proposition}
\label{proomegaL2}
It holds that
\begin{align*}
\sup_{0\leq t\leq T}\|w\|_2^2+\int_0^T\|\nabla w\|_2^2dt\leq C,
\end{align*}
where $C$ is a positive constant depending only $\|v_0\|_2$, $\|v_0\|_{4}$, $\|w_0\|_2$, $\alpha$, $\beta$, $\mu$, $\|K\|_{L^2(0, T;L^2)}$, and $T$.
\end{proposition}
\begin{proof}
Multiplying $\eqref{eqomega1}_1$ by $w$, it follows from integrating by parts, Lemma \ref{lem1a}, and the Minkowski, H\"{o}lder, Gagliardo-Nirenberg, and Young inequalities that
\begin{align*}
&\frac{1}{2}\frac{d}{dt}\|w\|_2^2+\mu\|\nabla w\|_2^2\\
=&-\int_\Omega\left(2\int_0^z\partial_x(v\partial_xv)d\xi+\partial_x(w v)+\alpha w+\beta \int_0^z\partial_xw d\xi+\int_0^zK_xd\xi\right)wdxdz\\
\leq & C\|v\partial_xv\|_2\|\partial_xw\|_2+|\alpha|\|w\|_2^2+|\beta|\|\partial_xw\|_2\|w\|_2+\|K\|_2\|\partial_xw\|_2
+\|w\|_4\|v\|_4\|\partial_xw\|_2\\
\leq & C\left(\|v\partial_xv\|_2\|\partial_xw\|_2+\|w\|_2^2+\|\partial_xw\|_2\|w\|_2+\|w\|_2^\frac{1}{2}\|w\|_{H^1}^\frac{1}{2}\|v\|_4\|\partial_x w\|_2\right)+\|K\|_2\|\partial_xw\|_2
\\
\leq& \frac{\mu}{2}\|\nabla w\|_2^2+C\left(1+\|v\|_4^4\right)\|w\|_2^2+C\left(\|v\partial_xv\|_2^2+\|K\|_2^2\right),
\end{align*}
from which, due to Proposition \ref{provL2} and Proposition \ref{provL4}, the conclusion follows by applying the Gr\"{o}nwall inequality.
\end{proof}
\begin{proposition}
\label{prouL2}
It holds
\begin{align*}
\sup_{0\leq t\leq T}\|\nabla v\|_2^2+\int_0^T\|\nabla^2 v\|_2^2dt\leq C,
\end{align*}
where $C$ is a positive constant depending only $\|v_0\|_{H^1}$, $\|w_0\|_2$, $\alpha$, $\beta$, $\mu$, $\|K\|_{L^2(0, T;L^2)}$, and $T$.
\end{proposition}
\begin{proof}
Multiplying $\eqref{eqpeini}_1$ by $-\triangle v$, using the boundary conditions \eqref{eq5}--\eqref{eq6}, it follows from integrating by parts and Lemma \ref{lem1a}, and the H\"{o}lder, Gagliardo-Nirenberg, and Young inequalities that
\begin{align*}
&\frac{1}{2}\frac{d}{dt}\|\nabla v\|_2^2+\mu\|\triangle v\|_2^2\\ \leq& C \left(\|v\partial_xv\|_2+\|v\|_2+\|\partial_xv\|_2+\|K\|_2\right)\|\triangle v\|_2+\|w\|_4\|\partial_zv\|_4\|\triangle v\|_2\\
\leq& C\left(\|v\partial_xv\|_2+\|v\|_2+\|\partial_xv\|_2+\|K\|_2\right)\|\triangle v\|_2+C\|w\|_2^\frac{1}{2}\|\nabla w\|_2^\frac{1}{2}\|\partial_zv\|_2^\frac{1}{2}\|\triangle v\|_2^\frac{3}{2}\\
\leq& \frac{\mu}{2}\|\nabla^2 v\|_2^2+C\|w\|_2^2\|\nabla w\|_2^2\|\nabla v\|_2^2+C\left(\|v\partial_xv\|_2^2+\|v\|_2^2+\|\partial_xv\|_2^2+\|K\|_2^2\right),
\end{align*}
from which, due to Propositions \ref{provL2}--\ref{proomegaL2}, and applying the Gr\"{o}nwall inequality and elliptic estimate, one obtains
\begin{align*}
&\|\nabla v\|_2^2(t)+\mu\int_0^t\|\nabla^2 v\|_2^2ds\\
\leq& e^{\int_0^t\|w\|_2\|\nabla w\|_2^2ds}\left[\|\nabla v_0\|_2^2+C\int_0^t\left(\|v\partial_xv\|_2^2+\|v\|_2^2+\|\partial_xv\|_2^2+\|K\|_2^2\right)ds\right],
\end{align*}
leading to the conclusion.
\end{proof}
Thanks to Propositions \ref{provL2}--\ref{prouL2}, one has
\begin{proposition}\label{partt}
It holds that
\begin{align*}
\int_0^T \|\partial_t v\|_2^2dt\leq C,
\end{align*}
where $C$ is a positive constant depending only $\|v_0\|_{H^1}$, $\|w_0\|_2$, $\alpha$, $\beta$, $\mu$, $\|K\|_{L^2(0, T;L^2)}$, and $T$.
\end{proposition}
\begin{proof}
Using $\eqref{eqpeini}_1$, the conclusion follows easily by Propositions \ref{provL2}--\ref{prouL2}.
\end{proof}

\section{The proof of the Theorem \ref{the1}}
\begin{proof}
It suffices to show the global existence as the uniqueness is guaranteed by the local well-posedness as in Proposition \ref{prolocal}. Applying Proposition \ref{prolocal} iteratively, one gets a unique strong solution $v$ up to the maximal time $T_*$ of existence.

It suffices to verify $T_*=+\infty$. Assume by contradiction that $T_*<+\infty$. Then, one has
\begin{align}
\overline{\lim_{t\rightarrow T_*^-}}\|v\|_{H^1}^2(t)=+\infty. \label{econt1s}
\end{align}
By Propositions \ref{provL2}--\ref{prouL2}, it holds that
\begin{align*}
\sup_{0\leq t\leq T}\|v\|_{H^1}^2+\int_0^T\|v\|_{H^2}dt\leq C,\quad \forall T\in (0, T_*),
\end{align*}
where $C$ is a positive constant independent of $T_*$. This contradicts to \eqref{econt1s}. Therefore $T_*=+\infty$ and $v$ is a global solution.
\end{proof}
\section*{Acknowledgments}

This work was supported in part by the National Nature Science Foundation of
China (Grant No. 12371204 and No. 12526408), the Key Project of National Nature Science Foundation
of China (Grant No. 12131010), and Guangdong Basic and Applied Basic Research
Foundation (Grant No. 2026A1515010778).

\end{document}